\documentclass[12pt]{article}
\usepackage[left=2.6cm,right=2.6cm,top=2.4cm,bottom=2.4cm]{geometry}
\usepackage{amsfonts}
\usepackage{titlesec}
\usepackage{cite,color,xcolor}
\usepackage{amsmath}
\usepackage{amsthm}
\usepackage{amssymb}
\usepackage{times}
\usepackage{bm}
\usepackage[colorlinks,citecolor=blue,urlcolor=blue]{hyperref}
\usepackage[utf8]{inputenc}
\usepackage{mathtools}
\usepackage{setspace}
\usepackage[numbers]{natbib}
\usepackage{cases}
\usepackage{fancyhdr}

\newtheorem{theorem}{Theorem}[section]

\newtheorem{lemma}{Lemma}[section]

\newtheorem{definition}{Definition}[section]
\newtheorem{remark}{Remark}[section]

\usepackage{txfonts}

\newcommand{\bal}{\begin{align}}
\newcommand{\bbal}{\begin{align*}}
\newcommand{\beq}{\begin{equation}}
\newcommand{\eeq}{\end{equation}}
\newcommand{\bca}{\begin{cases}}
\newcommand{\eca}{\end{cases}}

\newcommand{\pa}{\partial}
\newcommand{\fr}{\frac}

\newcommand{\De}{\Delta}

\newcommand{\dd}{\mathrm{d}}

\newcommand{\R}{\mathbb{R}}

\newcommand{\les}{\lesssim}

\newcommand\f{\left}
\newcommand\g{\right}

\begin{document}
\bibliographystyle{plain}
\title{Ill-posedness for the periodic Camassa--Holm type equations in the end-point critical Besov space $B^{1}_{\infty,1}$}

\author{Jinlu Li$^{1}$, Yanghai Yu$^{2,}$\footnote{E-mail: lijinlu@gnnu.edu.cn; yuyanghai214@sina.com(Corresponding author); mathzwp2010@163.com} and Weipeng Zhu$^{3}$\\
\small $^1$ School of Mathematics and Computer Sciences, Gannan Normal University, Ganzhou 341000, China\\
\small $^2$ School of Mathematics and Statistics, Anhui Normal University, Wuhu 241002, China\\
\small $^3$ School of Mathematics and Big Data, Foshan University, Foshan, Guangdong 528000, China}

\date{\today}

\maketitle\noindent{\hrulefill}

{\bf Abstract:} For the real-line case, it is shown that both the Camassa--Holm \cite{Guo} and Novikov equations \cite{Li-arx} are ill-posed in $B_{\infty,1}^{1}$. In this paper, by presenting a new construction of initial data which leads to the norm inflation phenomena, we prove that both the periodic Camassa--Holm and Novikov equations are also ill-posed in $B_{\infty,1}^{1}$.

{\bf Keywords:} Camassa--Holm and Novikov equation; Ill-posedness; Besov space.

{\bf MSC (2010):} 35B35, 37K10.
\vskip0mm\noindent{\hrulefill}

\thispagestyle{empty}
\section{Introduction}
In this paper, we consider the question of the well-posedness of the Cauchy problem
to a class of shallow water wave equations on the torus that containing the Camassa--Holm and Novikov equation. In order to elucidate the main ideas, our attention in this paper will be focused on the Camassa--Holm (CH) equation, which takes the form:
\begin{equation}\label{0}
\begin{cases}
u_t-u_{xxt}+3uu_x=2u_xu_{xx}+uu_{xxx},\quad (x,t)\in \mathbb{T}\times\R^+,\\
u(x,t=0)=u_0(x).
\end{cases}
\end{equation}
Here $\mathbb{T}=\R/2\pi\mathbb{Z}$, the scalar function $u = u(t, x)$ stands for the fluid velocity at time $t\geq0$ in the $x$ direction.
We can transform the CH equation equivalently into the following transport type equation
\begin{equation}\label{CH}
\begin{cases}
\partial_tu+u\pa_xu=-\pa_x\Lambda^{-2}\Big(u^2+\fr12(\pa_xu)^2\Big), \\
\Lambda^{-2}=\f(1-\pa^2_x\g)^{-1},\\
u(x,t=0)=u_0(x).
\end{cases}
\end{equation}
The CH equation appeared initially in the context of hereditary symmetries studied by Fuchssteiner and Fokas in \cite{Fokas} and then was derived explicitly as a water wave equation by Camassa and Holm \cite{Camassa}. Many aspects of the mathematical beauty of the CH equation have been exposed over the last two decades. Particularly, (CH) is completely integrable \cite{Camassa,Constantin-P} with a bi-Hamiltonian structure \cite{Constantin-E,Fokas} and infinitely many conservation laws \cite{Camassa,Fokas}. Also, it admits exact peaked
soliton solutions (peakons) of the form $u(x,t)=ce^{-|x-ct|}$ with $c>0,$ which are orbitally stable \cite{Constantin.Strauss}. Another remarkable feature of the CH equation is the wave breaking phenomena: the solution remains bounded while its slope becomes unbounded in finite time \cite{Constantin,Escher2,Escher3}. It is worth mentioning that the peaked solitons present the characteristic for the travelling water waves of greatest height and largest amplitude and arise as solutions to the free-boundary problem for incompressible Euler equations over a flat bed, see Refs. \cite{Constantin-I,Escher4,Escher5,Toland} for the details.
Due to these interesting and remarkable features, the CH equation has attracted much attention as a class of integrable shallow water wave equations in recent twenty years. Its systematic mathematical study was initiated in a series of papers by Constantin and Escher, see \cite{Escher1,Escher2,Escher3,Escher4,Escher5}.

Firstly, we recall the notion of well-posedness in the sense of Hadamard. We say that the Cauchy problem \eqref{0} is Hadamard (locally) well-posed in a Banach space $X$ if for any data $u_0\in X$ there exists (at least for a short time) $T>0$ and a unique solution in the space $\mathcal{C}([0,T),X)$ which depends continuously on the data. In particular, we say that the solution map is
continuous if for any $u_0\in X$, there exists a neighborhood $B \subset X$ of $u_0$ such that
for every $u \in B$ the map $u \mapsto U$ from $B$ to $\mathcal{C}([0, T]; X)$ is continuous, where $U$
denotes the solution to \eqref{0} with initial data $u_0$.
After the CH equation was derived physically in the context of water waves, there are
a large amount of literatures devoted to studying the well-posedness of the Cauchy problem \eqref{0} (see
Molinet's survey \cite{Molinet}). Particularly, the continuous dependence is rather important when PDEs are used to model
phenomena in the natural world since measurements are always associated with errors. Next we recall some progresses in this field.

{\bf Well-posedness.} Li and Olver \cite{li2000} proved that the Cauchy problem \eqref{0} is locally well-posed with the initial data $u_0\in H^s(\R)$ with $s > 3/2$ (see also \cite{GB}). Danchin \cite{d1} proved
the local existence and uniqueness of strong solutions to \eqref{0} with initial data in $B^s_{p,r}$ for $s > \max\{1 + 1/p , 3/2\}$ with $p\in[1,\infty]$ and $r\in[1,\infty)$. However, he \cite{d1} only obtained the continuity of the solution map of \eqref{0} with respect to the initial data in the space $\mathcal{C}([0, T ];B^{s'}_{p,r})$ with any $s'<s$. Li-Yin \cite{Li-Yin1} proved the continuity of the solution map of \eqref{0} with respect to the initial data in the space $\mathcal{C}([0, T];B^{s}_{p,r})$ with $r<\infty$. In particular, they \cite{Li-Yin1} proved that the solution map of \eqref{0} is weak continuous with respect to initial data $u_0\in B^s_{p,\infty}$. For the end-point critical case, Danchin \cite{d3} obtained the local well-posedness in the space $B^{3/2}_{2,1}$. Recently, Ye-Yin-Guo \cite{Ye} proved the uniqueness and continuous dependence of the Camassa--Holm type equations in critical Besov spaces $B^{1+1/p}_{p,1}$ with $p\in[1,\infty)$.

{\bf Ill-posedness.} When considering further continuous dependence, we proved the non-uniform dependence on initial data for \eqref{0} under both the framework of Besov spaces $B^s_{p,r}$ for $s>\max\big\{1+1/p, 3/2\big\}$ with $p\in[1,\infty], r\in[1,\infty)$ and $B^{3/2}_{2,1}$ in \cite{Li1, Li2} (see \cite{H-M,H-K,H-K-M} for earlier results in $H^s$). Danchin \cite{d3} obtained the ill-posedness of \eqref{0} in $B^{3/2}_{2,\infty}$ (the data-to-solution map is not continuous by using peakon solution). Byers \cite{Byers} proved that the Camassa--Holm equation is ill-posed in $H^s$ for $s < 3/2$ in the sense of norm inflation, which means that $H^{3/2}$ is the critical Sobolev space for well-posedness. In our recent papers \cite{Li22,Li22-jee}, we established the ill-posedness for the CH equation in $B^s_{p,\infty}(\mathbb{R})$ by proving the solution map to \eqref{0} starting from $u_0$ is discontinuous at $t = 0$ in the metric of $B^s_{p,\infty}(\mathbb{R})$. Moreover, for the real-line and torus cases, Guo-Liu-Molinet-Yin \cite{Guo-Yin} showed that the CH equation is ill-posed in $B_{p,r}^{1+1/p}(\mathbb{R}\;\text{or}\; \mathbb{T})$ with $(p,r)\in[1,\infty]\times(1,\infty]$ (especially in $H^{3/2}$) by proving the norm inflation. Particularly, for the end-point case $(p,r)=(\infty,1)$, Guo-Ye-Yin \cite{Guo} considered the real-line case and obtained the ill-posedness for the CH equation in $B^{1}_{\infty,1}(\R)$ by proving the norm inflation. We are concerned with the following natural and interesting question: Whether or not the periodic CH equation is ill-posed in $B^{1}_{\infty,1}(\mathbb{T})$? To the best of our knowledge, this is still an open problem. We shall solve this problem and present the negative result in this paper.
We can now state our main result as follows.
\begin{theorem}\label{th1}
For any $n\in \mathbb{Z}^+$ large enough, there exist $u_0$ with
\bbal
\|u_0\|_{B^1_{\infty,1}(\mathbb{T})}\leq \frac{1}{\log\log n}
\end{align*}
such that if we denote by $u\in \mathcal{C}([0,T);H^{\infty}(\mathbb{T}))$, the solution of the period Camassa--Holm equation with initial data $u_0$, then
\bbal
\|u(t_0)\|_{B^1_{\infty,1}(\mathbb{T})}\geq {\log\log n}\quad\text{with}\quad t_0\in \left(0,\frac{1}{\log n}\right].
\end{align*}
\end{theorem}

\begin{remark}
Since the norm inflation implies discontinuous of the data-to-solution map at the trivial function $u_0\equiv0$, Theorem \ref{th1} demonstrates the ill-posedness of the Camassa--Holm equation in $B^1_{\infty,1}(\mathbb{T})$ in the sense of  Hadamard.
\end{remark}
The Cauchy problem for the Novikov equation reads as (see \cite{Guo-Yin,Li1,Li2,Li-Yin1,Li22,Li22-jee,Li-arx,Yan1,Yan2,Ye} and the references therein)
\begin{equation}\label{N}
\begin{cases}
u_t+u^2u_x=-\frac12\Lambda^{-2}u_x^3-\pa_x\Lambda^{-2}\left(\frac32uu^2_x+u^3\right),\\
u(x,t=0)=u_0(x).
\end{cases}
\end{equation}
Home-Wang \cite{Home2008} proved that the Novikov equation  with cubic nonlinearity shares similar properties with the CH equation, such as a Lax pair in matrix form, a bi-Hamiltonian structure, infinitely many conserved quantities and peakon solutions given by the formula $u(x, t)=\sqrt{c}e^{-|x-ct|}$. The local well-posedness of the Novikov equation with initial data in Sobolev spaces and Besov spaces was studied in \cite{H-H,ni,wy,Yan1,Yan2}.
We would like to mention that, for only real-line case, Guo-Liu-Molinet-Yin \cite{Guo-Yin} proved that the Novikov  equation \eqref{N} is ill-posed in $B_{p,r}^{1+1/p}(\mathbb{R})$ with $(p,r)\in[1,\infty]\times(1,\infty]$ by proving the norm inflation. The only left an end-point case $r = 1$ for the Novikov equation in the real-line has been solved in our recent work \cite{Li-arx}. In this paper, we shall consider the torus case and establish the following
\begin{theorem}\label{th2}
For any $n\in \mathbb{Z}^+$ large enough, there exist $u_0$ with
\bbal
\|u_0\|_{B^1_{\infty,1}(\mathbb{T})}\leq \frac{1}{\log\log n}
\end{align*}
such that if we denote by $u\in \mathcal{C}([0,T);H^{\infty}(\mathbb{T}))$, the solution of the period Novikov equation with initial data $u_0$, then
\bbal
\|u(t_0)\|_{B^1_{\infty,1}(\mathbb{T})}\geq {\log\log n}\quad\text{with}\quad t_0\in \left(0,\frac{1}{\log n}\right].
\end{align*}
\end{theorem}
\noindent\textbf{Organization of our paper.} In Section \ref{sec2}, we list some notations and known results and recall some Lemmas which will be used in the sequel. In Section \ref{sec3} we present the proof of Theorem \ref{th1} by dividing it into several parts:
(1) Construction of initial data;
(2) Key Estimation for Discontinuity;
(3) The Equation Along the Flow;
(4) Norm inflation.
In Section \ref{sec4}, we present the constructions and estimations of initial data and leave the proof Theorem \ref{th2} to the interested readers.

\section{Preliminaries}\label{sec2}

\quad We define the periodic Fourier transform $\mathcal{F}_{\mathbb{T}}: \mathcal{D}(\mathbb{T})\rightarrow \mathcal{S}(\mathbb{Z})$  as
$$(\mathcal{F} u)(\xi)=\widehat{u}(\xi)=\int_{\mathbb{T}} e^{-\mathrm{i} x \xi} u(x) \mathrm{d} x.$$
and the inverse Fourier transform $\mathcal{F}^{-1}_{\mathbb{T}}: \mathcal{S}(\mathbb{Z}) \rightarrow  \mathcal{D}(\mathbb{T})$  as
$$(\mathcal{F}^{-1}u) (x)=\frac{1}{2 \pi} \sum_{\xi \in \mathbb{Z}} u(\xi) e^{\mathrm{i} x \xi}.$$
We decompose $u \in \mathcal{D}(\mathbb{T})$ on the circle $\mathbb{T}$ into Fourier series, i.e.
$$
u(x)=\frac{1}{2 \pi} \sum_{\xi \in \mathbb{Z}} \widehat{u}(\xi) e^{\mathrm{i} x \xi}.
$$
We are interested in solutions which take values in the Besov space $B_{p, r}^s(\mathbb{T})$. Recall that one way to define this space requires a dyadic partition of unity. Given a smooth bump function $\chi$ supported on the ball of radius $4 / 3$, and equal to 1 on the ball of radius $3 / 4$, we set $\varphi(\xi)=\chi(2^{-1} \xi)-\chi(\xi)$ and $\varphi_j(\xi)=\varphi(2^{-j} \xi)$. Using this partition, we define the periodic dyadic blocks as follows
\begin{align*}
&\Delta_j u=0, \quad \text{ if } \quad j \leq-2, \\
&\Delta_{-1} u=\mathcal{F}_x^{-1} \chi \mathcal{F}_x u=\frac{1}{2 \pi} \sum_{\xi \in \mathbb{Z}} \chi(\xi) \widehat{u}(\xi) e^{\mathrm{i} x \xi}, \\
&\Delta_j u=\mathcal{F}_x^{-1} \varphi(2^{-j} \xi) \mathcal{F}_x u=\frac{1}{2 \pi} \sum_{\xi \in \mathbb{Z}} \varphi_j( \xi)\widehat{u}(\xi) e^{\mathrm{i} x \xi}, \quad \text { if } \quad j \geq 0 .
\end{align*}
The low-frequency cut-off operator $S_j$ is defined as follows
$$
{S}_ju=\sum_{-1\leq k\leq j-1}{\Delta}_ku,\quad \forall j\geq-1.
$$
Therefore, we obtain the Littlewood-Paley decomposition of $u$
$$u=\sum_{j \in \mathbb{Z}} \Delta_j u\quad \text{in}\quad \mathcal{S}^{\prime}(\mathbb{T}).$$
We have the following useful facts
\begin{itemize}
  \item $\Delta_k \Delta_j u \equiv 0, \quad \text { if } \quad|k-j| \geq 2$,
  \item $\Delta_j(S_{k-1} u \Delta_k v) \equiv 0, \quad \text { if } \quad|k-j| \geq 5, \quad \forall u, v \in \mathcal{S}^{\prime}(\mathbb{T})$,
  \item $\|\Delta_j u\|_{L^p} \leq C\|u\|_{L^p} \quad\text{and}\quad\|S_j u\|_{L^p} \leq C\|u\|_{L^p}, \quad \forall p \in[1, \infty]$,
\end{itemize}
where $C$ is a positive constant independent of $j$.

\begin{definition}[Besov space]
Let $s\in\mathbb{R}$ and $(p,r)\in[1, \infty]^2$. The nonhomogeneous Besov space $B^{s}_{p,r}(\mathbb{T})$ is defined by
\begin{align*}
B^{s}_{p,r}(\mathbb{T}):=\Big\{f\in \mathcal{S}'(\mathbb{T}):\;\|f\|_{B^{s}_{p,r}(\mathbb{T})}<\infty\Big\},
\end{align*}
where
\begin{numcases}{\|f\|_{B^{s}_{p,r}(\mathbb{T})}=}
\left(\sum_{j\geq-1}2^{sjr}\|\Delta_jf\|^r_{L^p(\mathbb{T})}\right)^{\fr1r}, &if $1\leq r<\infty$,\nonumber\\
\sup_{j\geq-1}2^{sj}\|\Delta_jf\|_{L^p(\mathbb{T})}, &if $r=\infty$.\nonumber
\end{numcases}
\end{definition}
The operators $\Delta_j$ defined on the periodic domain share many properties with those
on the whole space(see \cite{B}). Particularly, we can derive the Bernstein's inequality and commutator estimate for the periodic functions.
\begin{lemma}\label{s} Let $\alpha \geq 0$ and $1 \leq q \leq p \leq \infty$.
There exists a constant $C>0$ such that
$$
\left\|\Delta_j \partial_x^\alpha f\right\|_{L^p\left(\mathbb{T}\right)} \leq C 2^{\alpha j+j \left(\frac{1}{q}-\frac{1}{p}\right)}\|\Delta_j f\|_{L^q\left(\mathbb{T}\right)}.
$$
\end{lemma}
\begin{lemma}\label{lem2.2}
Let $1 \leq r \leq \infty$, $1 \leq p \leq p_{1} \leq \infty$ and $\frac{1}{p_{2}}=\frac{1}{p}-\frac{1}{p_{1}}$. There exists a constant $C$ depending continuously on $p$ and $p_1$ such that
$$
\left\|2^{j}\|[\Delta_{j},v]\pa_x f\|_{L^{p}(\mathbb{T})}\right\|_{\ell^{r}(j\geq-1)} \leq C\left(\|\pa_x v\|_{L^{\infty}(\mathbb{T})}\|f\|_{B_{p, r}^{1}(\mathbb{T})}+\|\pa_x f\|_{L^{p_{2}}(\mathbb{T})}\|\pa_x v\|_{B_{p_1,r}^{0}(\mathbb{T})}\right),
$$
where we denote the standard commutator $[\Delta_j,v]\pa_xf=\Delta_j(v\pa_xf)-v\Delta_j\pa_xf$.
\end{lemma}

Let us complete this section by proving the simple fact which will be used often in the sequel.
\begin{lemma}\label{cos}
For any $\lambda\in \mathbb{Z}$, we have
\begin{align*}
\mathcal{F}[\cos (\lambda x)](\xi)
&=
\begin{cases}
\pi, & \xi=\pm\lambda, \\
0, & \xi \neq \pm \lambda,
\end{cases}
\end{align*}
and
\begin{align*}
\mathcal{F}[\sin (\lambda x)](\xi)
&=
\begin{cases}
-\mathrm{i}\pi, & \xi=\lambda, \\
\mathrm{i}\pi, & \xi=-\lambda, \\
0, & \xi \neq \pm \lambda.
\end{cases}
\end{align*}
\end{lemma}
\begin{proof} An obvious computation gives that
\begin{align*}
&\mathcal{F}[\cos (\lambda x)](\xi)
=\frac{1}{2}  \int_{-\pi}^{\pi} \left(e^{-\mathrm{i}(\xi-\lambda) x} + e^{-\mathrm{i}(\xi+\lambda) x}\right) \dd x,\\
&\mathcal{F}[\sin (\lambda x)](\xi)=\frac{\mathrm{i}}{2}  \int_{-\pi}^{\pi}\left(e^{-\mathrm{i}(\xi+\lambda) x}- e^{-\mathrm{i}(\xi-\lambda) x}\right) \dd x,
\end{align*}
which implies the desired result of Lemma \ref{cos}.
\end{proof}
\section{Proof of Theorem \ref{th1}: Camassa--Holm equation}\label{sec3}
In this section, we prove Theorem \ref{th1}.
\subsection{Construction of Initial Data}\label{sec3.1}
\quad From now on, we set
\bbal
n\in 16\mathbb{N}=\left\{16,32,48,\cdots\right\}\quad\text{and}\quad
\mathbb{N}(n)=\left\{k\in \mathbb{N}: \frac{n}8 \leq k\leq \frac{n}4\right\}.
\end{align*}
  We introduce the following new notation which will be used often throughout this paper
  \bbal
  \|f\|_{B^k_{\infty,1}\left(\mathbb{N}(n)\right)}=\sum_{j\in\mathbb{N}(n)}2^{kj}\|\Delta_jf\|_{L^\infty},\quad k\in\{0,1\}.
  \end{align*}
Let
\bbal
h(x)=
\bca
h(x+2\pi), & x\in \R,\\
-\frac12, & x\in(-\pi,0),\\
\frac12, & x\in(0,\pi),\\
0, & x\in\{-\pi,0,\pi\}.
\eca
\end{align*}
It is straightforward to calculate the Fourier series of $h(x)$
\bbal
h(x)= \frac2\pi\sum^{\infty}_{k=1}\frac{\sin[(2k-1)x]}{2k-1}.
\end{align*}
Then we have the following crucial estimation.
\begin{lemma}\label{le-e0}
There exists two positive constants $c_1,c_2$ independent of $n$ such that
\bbal
c_1\leq\|\Delta_jh(x)\|_{L^\infty}\leq c_2,\quad j\in \mathbb{N}(n).
\end{align*}
\end{lemma}

\begin{proof}
Since $\varphi_j(\xi)$ is symmetric, i.e., $\varphi_j(\xi)=\varphi_j(|\xi|)$, using Lemma \ref{cos}, we have
\bal\label{h}
\Delta_jh(x)
&= \frac{1}{2 \pi} \sum_{\xi \in \mathbb{Z}} \varphi_j(\xi)\widehat{h}(\xi) e^{\mathrm{i} x \xi}\nonumber\\
&= -\frac{\mathrm{i}}{\pi}\sum_{\xi\in\mathbb{Z}}\sum^{\infty}_{k= 1}\varphi_j(\xi)\frac{1}{2k-1}\mathbf{1}_{2k-1}(|\xi|)\mathrm{sign}(\xi)e^{\mathrm{i}x\xi}
\nonumber\\
&=-\frac{\mathrm{i}}{\pi}\sum_{\fr342^{j}\leq |\xi|\leq \fr832^{j}}\sum^{\infty}_{k= 1}\varphi_j(\xi)\frac{1}{2k-1}\mathbf{1}_{2k-1}(|\xi|)\mathrm{sign}(\xi)e^{\mathrm{i}x\xi}
\nonumber\\
&=\frac{2}{\pi}\sum_{\fr342^{j}\leq \xi\leq \fr832^{j}}\sum^{\infty}_{k= 1}\varphi_j(\xi)\frac{1}{2k-1}\mathbf{1}_{2k-1}(\xi)\sin(x\xi),
\end{align}
where $\mathbf{1}_{m}(x)$ is the indicator function, taking a value of 1 if $x=m$ and 0 otherwise.

It follows from \eqref{h}  that
\bbal
c_2\geq\|\Delta_jh(x)\|_{L^\infty}&\geq \frac{2}{\pi}\left|\sum_{\fr342^{j}\leq \xi\leq \fr832^{j}}\sum^{\infty}_{k= 1}\varphi_j(\xi)\frac{1}{2k-1}\mathbf{1}_{2k-1}(\xi)\sin(x\xi)\right|_{x=2^{-j-2}\pi}\\
&\geq \frac{2}{\pi}\sin\frac{3\pi}{16}\sum_{\fr342^{j}\leq 2k-1\leq \fr832^{j}}\varphi_j(2k-1)\frac{1}{2k-1}
\\&\geq \frac{2}{\pi}\sin\frac{3\pi}{16} \sum_{\fr432^{j}\leq 2k-1\leq \fr322^{j} }\frac{1}{2k-1}\thickapprox c_1>0.
\end{align*}
Then we  complete the proof of Lemma \ref{le-e0}.
\end{proof}
Now, we can define the initial data $u_{0,n}$ for the Camassa--Holm equation
\bbal
u_{0,n}(x)&= 2^{-n}n^{-\frac25}\log n\cdot\cos(2^nx)\left(1+n^{-\frac15}f_n(x)\right),
\end{align*}
where we denote the low frequency part of $h(x)$ by
\bal\label{f}
f_n(x):=S_{\frac{n}2}h(x).
\end{align}

\subsection{Key Estimation for Discontinuity}
The following two Lemmas play an important role in the proof of Theorem \ref{th1}.
\begin{lemma}\label{le-e1}
There exists a positive constant $C$ independent of $n$ such that
\bbal
&2^{n}\|u_{0,n}\|_{L^\infty}+\|\pa_xu_{0,n}\|_{L^\infty}
\leq Cn^{-\frac25}\log n,\\
&\|u_{0,n}\|_{B^1_{\infty,1}}\leq C n^{-\frac25}\log n.
\end{align*}
In particular, it holds that
$$\|u_{0,n}\|_{C^{0,1}}=\|u_{0,n}\|_{L^\infty}+\|\pa_xu_{0,n}\|_{L^\infty}\leq C n^{-\frac25}\log n.$$
\end{lemma}
\begin{proof}
By the construction of $u_{0,n}$, one has
\bal\label{u}
\pa_xu_{0,n}&=-n^{-\frac25}\log n\cdot \left[\sin(2^nx)\left(1+n^{-\frac15}f_n\right)-2^{-n}n^{-\frac15}\cos(2^nx)\pa_xf_n\right].
\end{align}
Using Bernstein's inequality, we have
\bbal
2^{n}\|u_{0,n}\|_{L^\infty}+\|\pa_xu_{0,n}\|_{L^\infty}
&\leq Cn^{-\frac25}\log n\left(1+n^{-\frac15}\|f_n\|_{L^\infty}+2^{-n}n^{-\frac15}\|\pa_xf_n\|_{L^\infty}\right)\leq Cn^{-\frac25}\log n.
\end{align*}
The Fourier transform of $\sin(2^nx)$, $\sin(2^{n}x)f_n$ and $\cos(2^{n}x)\pa_xf_n$  are supported in annulus $\{\xi:|\xi|\sim 2^n\}$, then we have
\bbal
\De_j\left(u_{0,n}\right)=0,\quad \text{for}\quad j\notin\{n-1,n,n+1\},
\end{align*}
which tells us that
\bbal
\|u_{0,n}\|_{B^1_{\infty,1}}= \sum_{j\geq-1} 2^j\|\Delta_ju_{0,n}\|_{L^\infty}\leq C2^n\|u_{0,n}\|_{L^\infty}
&\leq C n^{-\frac25}\log n.
\end{align*}
This completes the proof of Lemma \ref{le-e1}.
\end{proof}
\begin{lemma}\label{le-e2}
There exists a positive constant $c$ independent of $n$ such that
\bbal
\left\|\left(\pa_xu_{0,n}\right)^2\right\|_{B^0_{\infty,1}\left(\mathbb{N}(n)\right)}\geq c(\log n)^2, \qquad n\gg1.
\end{align*}
\end{lemma}
\begin{proof} Due to \eqref{u}, one has
\bbal
n^{\frac45}(\log n)^{-2}\cdot (\pa_xu_{0,n})^2&=I_1-I_2+I_3,
\end{align*}
where
\bbal
I_1&=\sin^2(2^nx)\left(1+n^{-\frac15}f_n\right)^2,\\
I_2&=2^{-n}n^{-\frac15}\sin(2^{n+1}x)\pa_xf_n\left(1+n^{-\frac15}f_n\right),\\
I_3&=2^{-2n}n^{-\frac25}\cos^2(2^nx)\left(\pa_xf_n\right)^2.
\end{align*}
Notice that the Fourier transform of $I_2$ is supported in annulus $\{\xi:|\xi|\sim 2^n\}$, it follows that
$$\Delta_j\left(I_2\right)=0\quad\text{for}\quad j\in \mathbb{N}(n),$$
which gives directly that
\bal\label{li1}
\left\|I_2\right\|_{B^0_{\infty,1}\left(\mathbb{N}(n)\right)}=0.
\end{align}
Using the following rough estimate
\bbal
\left\|\cos^2(2^nx)\left(\pa_xf_n\right)^2\right\|_{B^0_{\infty,1}\left(\mathbb{N}(n)\right)}\leq Cn\|\pa_xf_n\|^2_{L^\infty}\leq Cn2^{n},
\end{align*}
we obtain
\bal\label{li2}
\left\|I_3\right\|_{B^0_{\infty,1}\left(\mathbb{N}(n)\right)}\leq Cn^{\frac35}2^{-n}.
\end{align}
To estimate the first term $I_1$, by the simple equality $2\sin^2(a)=1-\cos(2a)$, we rewrite it as
\bbal
I_1&=\frac{1}{2}\left(1-\cos(2^{n+1}x)\right)-\frac{1}{2}n^{-\frac25}\cos(2^{n+1}x)(f_n)^2-n^{-\frac15}\cos(2^{n+1}x)f_n\\
&\quad+n^{-\frac15} f_n+\frac{1}{2}n^{-\frac25}f^2_n.
\end{align*}
Notice that
$$\Delta_j\left(1-\cos(2^{n+1}x)\right)=\Delta_j\left(\cos(2^{n+1}x)f_n\right)=\Delta_j\left(\cos(2^{n+1}x)(f_n)^2\right)=0 \quad\text{for}\quad j\in \mathbb{N}(n),$$
then we have
\bal\label{y}
\left\|I_1\right\|_{B^0_{\infty,1}\left(\mathbb{N}(n)\right)}
&\geq n^{-\frac15}
\left\|f_n\right\|_{B^0_{\infty,1}\left(\mathbb{N}(n)\right)}-\frac{1}{2}n^{-\frac25}\left\|f^2_n\right\|_{B^0_{\infty,1}\left(\mathbb{N}(n)\right)}.
\end{align}
For the first term, using Lemma \ref{le-e0} yields
\bal\label{y1}
&\left\|f_n\right\|_{B^0_{\infty,1}\left(\mathbb{N}(n)\right)}=\sum_{j\in \mathbb{N}(n)} \|\Delta_jf_n\|_{L^\infty}=\sum_{j\in \mathbb{N}(n)} \|\Delta_jh\|_{L^\infty}\approx n.
\end{align}
For the second term, we have
\bal\label{y2}
&\left\|f^2_n\right\|_{B^0_{\infty,1}\left(\mathbb{N}(n)\right)}\leq Cn\|f_n\|^2_{L^\infty}\leq Cn.
\end{align}
Inserting \eqref{y1}-\eqref{y2} into \eqref{y}, then we have for large $n$ enough
\bal\label{li3}
\left\|I_1\right\|_{B^0_{\infty,1}\left(\mathbb{N}(n)\right)}
&\geq cn^{\frac45}.
\end{align}
Thus, combining \eqref{li1}-\eqref{li2} and \eqref{li3}, we deduce that for large $n$ enough
\bbal
&n^{\frac45}(\log n)^{-2}\cdot \left\|(\pa_xu_{0})^2\right\|_{B^0_{\infty,1}\left(\mathbb{N}(n)\right)}\geq cn^{\fr45},
\end{align*}
which is nothing but the desired result of Lemma \ref{le-e2}.
\end{proof}
\subsection{The Equation Along the Flow}\label{sec3.2}

Given a Lipschitz velocity field $u$, we may solve the following ODE to find the flow induced by $u$:
\begin{align}\label{ode}
\quad\begin{cases}
\frac{\dd}{\dd t}\phi(t,x)=u(t,\phi(t,x)),\\
\phi(0,x)=x,
\end{cases}
\end{align}
which is equivalent to the integral form
\bal\label{n}
\phi(t,x)=x+\int^t_0u(\tau,\phi(\tau,x))\dd \tau.
\end{align}
Considering the equation
\begin{align}\label{pde}
\quad\begin{cases}
\pa_tv+u\pa_xv=P,\\
v(0,x)=v_0(x),
\end{cases}
\end{align}
then, we get from \eqref{pde} that
\bbal
\pa_t(\De_jv)+u\pa_x\De_jv&=R_j+\Delta_jP,
\end{align*}
with $R_j=[u,\De_j]\pa_xv=u\De_j\pa_xv-\Delta_j(u\pa_xv)$.

Due to \eqref{ode}, then
\bbal
\frac{\dd}{\dd t}\left((\De_jv)\circ\phi\right)&=R_j\circ\phi+\Delta_jP\circ\phi,
\end{align*}
which means that
\bal\label{l6}
\De_jv\circ\phi=\De_jv_0+\int^t_0R_j\circ\phi\dd \tau+\int^t_0\Delta_jP\circ\phi\dd \tau.
\end{align}
\subsection{Norm Inflation}\label{sec3.3}
Following the proof of Lemma 3.26 in \cite{B}, we can obtain
$$\|u(t)\|_{C^{0,1}(\mathbb{T})}\leq \|u_{0,n}\|_{C^{0,1}(\mathbb{T})}\exp\left(\widetilde{C}\int_0^t \|\pa_xu(\tau)\|_{L^{\infty}(\mathbb{T})}\mathrm{d}\tau\right),$$
which implies that for all $t\in\left(0,\min\left\{1,1/(2\widetilde{C}\|u_{0,n}\|_{C^{0,1}})\right\}\right]$
\begin{align*}
&\|u(t)\|_{C^{0,1}}\leq C\|u_{0,n}\|_{C^{0,1}}.
\end{align*}
For $n\gg1$, using Lemma \ref{le-e1}, we have for $t\in[0,1]$
\bbal
\|u\|_{C^{0,1}}\leq C\|u_{0,n}\|_{C^{0,1}}\leq Cn^{-\frac25}\log n.
\end{align*}
To prove Theorem \ref{th1}, it suffices to show that there exists $t_0\in\left(0,\frac{1}{\log n}\right]$ such that
\bal\label{m}
\|u(t_0,\cdot)\|_{B^1_{\infty,1}}\geq \log\log n.
\end{align}
We prove \eqref{m} by contradiction. If \eqref{m} were not true, then
\bal\label{m1}
\sup_{t\in\left(0,\frac{1}{\log n}\right]}\|u(t,\cdot)\|_{B^1_{\infty,1}}< \log\log n.
\end{align}
We divide the proof into two steps.

{\bf Step 1: Lower bounds for $(\De_ju)\circ \phi$}

Now we consider the equation along the Lagrangian flow-map associated to $u$.
Utilizing \eqref{l6} to \eqref{CH} yields
\bbal
(\De_ju)\circ \phi&=\De_ju_{0,n}+\int^t_0R^1_j\circ \phi\dd \tau +\int^t_0\De_jF\circ \phi\dd \tau
\\&\quad +\int^t_0\big(\De_jE\circ \phi-\De_jE_0\big)\dd \tau+t\De_jE_0,
\end{align*}
where
\bbal
&R^1_j=[u,\De_j]\pa_xu,
\quad\quad F=-\pa_x\Lambda^{-2}u^2, \\
&E=-\frac12\pa_x\Lambda^{-2}(\pa_xu)^2\quad\text{with}\quad E_0=-\frac12\pa_x\Lambda^{-2}(\pa_xu_{0,n})^2.
\end{align*}
Due to Lemma \ref{le-e2}, we deduce
\bal\label{g1}
\sum_{j\in \mathbb{N}(n)}2^j\|\De_jE_0\|_{L^{\infty}}
&\approx \sum_{j\in \mathbb{N}(n)}\|\De_j\pa_xE_0\|_{L^{\infty}}
\geq c\sum_{j\in \mathbb{N}(n)}\|\De_j(\pa_xu_{0,n})^2\|_{L^{\infty}}
\geq c(\log n)^2.
\end{align}
Notice that the fact
$\|f(t,\phi(t,\cdot))\|_{L^\infty}= \|f(t,\cdot)\|_{L^\infty},$ using the commutator estimate from Lemma \ref{lem2.2}, we have
\bal\label{g2}
\sum_{j\geq -1}2^j\|R^1_j\circ \phi\|_{L^{\infty}}&\leq C\sum_{j\geq -1}2^j\|R^1_j\|_{L^{\infty}}
\leq C\|\pa_xu\|_{L^\infty}\|u\|_{B^1_{\infty,1}}
\leq C n^{-\frac25}(\log n)^2.
\end{align}
By the fundamental theorem of calculus in the time variable, we have
\bbal
\|u(t)-u_{0,n}\|_{L^\infty}&\leq\int^t_0\|\pa_\tau u\|_{L^\infty} \dd\tau
\nonumber\\
&\les\int^t_0\|u\pa_xu\|_{L^\infty}\dd \tau
+ \int^t_0\left\|\pa_x\Lambda^{-2}\left(u^2+\fr12(\pa_xu)^2\right)\right\|_{L^\infty}\dd \tau
\nonumber\\
&\les\int^t_0\|u\|^2_{C^{0,1}}\dd \tau\nonumber\\
&\les t\|u_{0,n}\|^2_{C^{0,1}}.
\end{align*}
Then, we have for $t\in\left(0,\frac{1}{\log n}\right]$
\bbal
2^j\|\De_jF\circ \phi\|_{L^{\infty}}&\leq C2^j\|\De_jF\|_{L^{\infty}}\leq C\|u\|_{L^\infty}^2
\leq C\left(\|u_{0,n}\|_{L^\infty}+\|u_{0,n}\|^2_{C^{0,1}}\right)^2
\leq Cn^{-\frac85}(\log n)^4,
\end{align*}
which implies
\bal\label{g3}
\sum_{j\in \mathbb{N}(n)}2^j\|\De_jF\circ \phi\|_{L^{\infty}}\leq C n^{-\frac35}(\log n)^4.
\end{align}
Combining \eqref{g1}-\eqref{g3} and using Lemmas \ref{le-e1}-\ref{le-e2} yields
\bbal
&\quad\sum_{j\in \mathbb{N}(n)}2^j\|(\De_ju)\circ \phi\|_{L^{\infty}}\\
&\geq t\sum_{j\in \mathbb{N}(n)}2^j\|\De_jE_0\|_{L^{\infty}}-\sum_{j\in \mathbb{N}(n)}2^j\|\De_jE\circ \phi-\De_jE_0\|_{L^{\infty}}-
Cn^{-\frac25}(\log n)^4
-C\|u_{0,n}\|_{B^1_{\infty,1}}
\\&\geq ct(\log n)^2-\sum_{j\in \mathbb{N}(n)}2^j\|\De_jE\circ \phi-\De_jE_0\|_{L^{\infty}}-
Cn^{-\frac25}(\log n)^4.
\end{align*}
{\bf Step 2: Upper bounds for $\De_jE\circ \phi-\De_jE_0$}

By easy computations, we can see that
\begin{align}\label{E}
\quad\begin{cases}
\pa_tE+u\pa_xE=G,\\
E(0,x)=E_0=-\frac12\pa_x\Lambda^{-2}(\pa_xu_{0,n})^2,
\end{cases}
\end{align}
where
\bbal
G=&\frac13u^3-\frac12u\Lambda^{-2}(\pa_xu)^2
-\Lambda^{-2}\left(\frac13u^3-
\frac12u(\pa_xu)^2-\pa_x\Big(\pa_xu\Lambda^{-2}
(u^2+\frac12(\pa_xu)^2)\Big)\right).
\end{align*}
Utilizing \eqref{l6} to \eqref{E} yields
\bbal
\De_jE\circ \phi-\De_jE_0=\int^t_0[u,\De_j]\pa_xE\circ \phi\dd \tau +\int^t_0\De_jG\circ \phi\dd \tau.
\end{align*}
Using the commutator estimate from Lemma \ref{lem2.2}, one has
\bbal
2^j\|[u,\De_j]\pa_xE\|_{L^{\infty}}\leq C(\|\pa_xu\|_{L^\infty}\|E\|_{B^1_{\infty,\infty}}
+\|\pa_xE\|_{L^\infty}\|u\|_{B^1_{\infty,\infty}})\leq C\|u\|^3_{C^{0,1}}
\end{align*}
and
\bbal
2^j\|\De_jG\|_{L^{\infty}}\leq C \|u\|^3_{C^{0,1}}.
\end{align*}
Then, we have
\bbal
2^j\|\De_jE\circ \phi-\De_jE_0\|_{L^{\infty}}\leq C\|u\|^3_{C^{0,1}}\leq C\|u_{0,n}\|^3_{C^{0,1}}\leq Cn^{-\fr65}(\log n)^3,
\end{align*}
which leads to
\bbal
\sum_{j\in \mathbb{N}(n)}2^j\|\De_jE\circ \phi-\De_jE_0\|_{L^{\infty}}
\leq Cn^{-\fr15}(\log n)^3.
\end{align*}
Combining Step 1 and Step 2, then for $t=\frac{1}{\log n}$, we obtain for $n\gg1$
\bbal
\|u(t)\|_{B^1_{\infty,1}}&\geq \|u(t)\|_{B^1_{\infty,1}(\mathbb{N}(n))}\geq C\sum_{j\in \mathbb{N}(n)}2^j\|(\De_ju)\circ \phi\|_{L^{\infty}}
\\&\geq c t(\log n)^2-Cn^{-\frac15}(\log n)^3-
Cn^{-\frac25}(\log n)^4\\
&\geq \log\log n,
\end{align*}
which contradicts the hypothesis \eqref{m1}.

In conclusion, we obtain the norm inflation and hence the ill-posedness of the CH equation. Thus, Theorem \ref{th1} is proved.

\section{Proof of Theorem \ref{th2}: Novikov equation}\label{sec4}
For the Novikov equation, we have to modify the construction of initial data as follows
\bal\label{vv}
v_{0,n}(x)&=2^{-n}n^{-\frac14}\log n\cdot \cos(2^nx)\left(1+n^{-\frac14}f_n(x)\right)+n^{-\frac14},
\end{align}
where $f_n$ is defined by \eqref{f}.

Then we have
\begin{lemma}\label{le-e1v}
There exists some positive constants $C$ and $c$ independent of $n$ such that for $n\gg1$
\bbal
&2^{n}\|v_{0,n}\|_{L^\infty}+\|\pa_xv_{0,n}\|_{L^\infty}
\leq Cn^{-\frac14}\log n,\\
&\|v_{0,n}\|_{B^1_{\infty,1}}\leq C n^{-\frac14}\log n,\\
&\left\|v_{0,n}\left(\pa_xv_{0,n}\right)^2\right\|_{B^0_{\infty,1}\left(\mathbb{N}(n)\right)}\geq c(\log n)^2.
\end{align*}
\end{lemma}
\begin{proof} Due to \eqref{vv}, one has
\bbal
\pa_xv_{0,n}&=-n^{-\frac14}\log n\cdot \left[\sin(2^nx)\left(1+n^{-\frac14}f_n\right)-2^{-n}n^{-\frac14}\cos(2^nx)\pa_xf_n\right],
\end{align*}
which gives that
\bal\label{v1}
v_{0,n}\left(\pa_xv_{0,n}\right)^2&=n^{-\frac34}\log^2 n\cdot \sin^2(2^nx)\left(1+n^{-\frac14}f_n\right)^2+\text{Remaining Terms}.
\end{align}
By identical reasoning to Lemma \ref{le-e1}, we complete the proof of Lemma \ref{le-e1v}.
\end{proof}
With the aid of Lemma \ref{le-e1v}, we can prove Theorem \ref{th2} by repeating the above procedure of subsections \ref{sec3.2} and \ref{sec3.3}. Since the process is standard, we skip the details here and refer to the line case in \cite{Li-arx}.
\section*{Acknowledgements}
J. Li is supported by the National Natural Science Foundation of China (11801090 and 12161004) and Jiangxi Provincial Natural Science Foundation (20212BAB211004 and 20224BAB201008). Y. Yu is supported by the National Natural Science Foundation of China (12101011). W. Zhu is supported by the National Natural Science Foundation of China (12201118) and Guangdong
Basic and Applied Basic Research Foundation (2021A1515111018).

\section*{Declarations}
\noindent\textbf{Data Availability} No data was used for the research described in the article.

\vspace*{1em}
\noindent\textbf{Conflict of interest}
The authors declare that they have no conflict of interest.

\end{document}